\documentclass[11pt]{amsart}

\usepackage{amsmath,amsthm} 
\usepackage {latexsym}
\usepackage{amssymb}
\usepackage{xcolor}

\usepackage[utf8]{inputenc}
\newcommand{\beal}{\begin{align}}
\newcommand{\enal}{\end{align}}
\newcommand{\bealn}{\begin{align*}}
\newcommand{\enaln}{\end{align*}}
\newcommand{\bear}{\begin{eqnarray}}
\newcommand{\eear}{\end{eqnarray}}
\newcommand{\beeq}{\begin{equation}}
\newcommand{\eneq}{\end{equation}}

\newcommand{\eps}{{\varepsilon}}
\newcommand{\R}{{\mathbb R}}

\def\bm{\left[ \begin{array}{cc}}
\def\endm{\end{array}\right]}

\def\bm{\left[\begin{matrix} }
\def\endm{\end{matrix}\right]}

\def\R{{\mathbb R}}

\newcommand{\calQ}{{\mathcal Q}}
\newcommand{\calV}{{\mathcal V}}
\newcommand{\calH}{{\mathcal H}}
\newcommand{\calVs}{{\mathcal V_{\sigma}}}
\newcommand{\calVa}{{\mathcal{V}'_{\sigma}}}

\newtheorem{theorem}{Theorem}

\newtheorem{defi}[theorem]{Definition}

\newtheorem{proposition}[theorem]{Proposition}

\newtheorem{remark}[theorem]{Remark}

\renewcommand{\epsilon}{\eps}
\renewcommand{\tilde}{\widetilde}
\numberwithin{equation}{section}
\numberwithin{theorem}{section}

\usepackage{xr-hyper} 
\usepackage[colorlinks]{hyperref}

\begin{document}
\title[Small-time  stabilization of  Navier--Stokes  equations]{Small-time local stabilization of the two dimensional incompressible Navier--Stokes  equations}
\author{Shengquan Xiang}
\address{Bâtiment des Mathématiques, EPFL\\
 Station 8, CH-1015 Lausanne, Switzerland}
\email{\texttt{shengquan.xiang@epfl.ch.}}
\begin{abstract} 
We provide  explicit time-varying  feedback laws that locally stabilize the two dimensional internal controlled  incompressible Navier--Stokes equations in arbitrarily small time. We also obtain quantitative  rapid stabilization  via stationary feedback laws, as well as quantitative null controllability  with explicit controls having  $e^{C/T}$ costs.
\end{abstract}

\maketitle
\section{introduction}
Let $\Omega$ be a bounded connected open set  in $\mathbb{R}^2$ with smooth boundary.    Let the controlled domain  $\omega\subset \Omega$ be a nonempty  open subset.  We are interested in the stabilization and null controllability of the two dimensional incompressible  Navier-Stokes system with internal control,
\begin{gather}
\begin{cases}
y_t- \Delta y+ \left(y \cdot  \nabla\right) y+ \nabla p  = 1_{\omega} f  & \textrm{ in } \Omega,  \label{cons1}\\
\textrm{div } y= 0  & \textrm{ in } \Omega, \\
y=0 &\textrm{ on } \partial \Omega,
\end{cases}
\end{gather} 
where,  the state $y(t, \cdot)$ and  the control term $f(t, \cdot)$  are in $\calH$.   We adapt the standard fluid mechanics framework, 
\begin{gather*}
\calH:= \{y\in L^2(\Omega)^2: \; \textrm{div } y=0 \textrm{ in } \Omega, \; y\cdot n=0 \textrm{ on } \partial \Omega\},\\
\calV_{\sigma}:= \{y\in H^1_0(\Omega)^2: \; \textrm{div } y=0 \textrm{ in } \Omega\} \; \textrm{ and } \; \calV:= \{y\in H^1_0(\Omega)^2\},
\end{gather*}
with $\calVs\hookrightarrow \calV\hookrightarrow \calH\hookrightarrow L^2(\Omega)^2\hookrightarrow \calV'\hookrightarrow \calVa$.  We have taken the viscosity coefficient as 1 to simplify the presentation. 

When  dealing with  stabilization problems  the control term $f$ is regarded as a feedback control governed by some ``feedback  application" that depends on  current states and time, $U(t; y)$: 
\begin{equation}\label{cons-f}
f(t,x):= U(t; y(t, x)),
\end{equation}
where the application   $U$   is the so called \textit{time-varying feedback law},
\begin{equation}\label{cons2}
\left\{\begin{array}{cccc}
U:&  \mathbb{R}\times \calH &\to &   \calH
\\
&(t; y)&\mapsto & U(t; y).
\end{array}
\right.
\end{equation}
The \textit{closed-loop system} associated to  the preceding feedback law $U$ is the evolution system \eqref{cons1}--\eqref{cons2}.   A \textit{stationary} feedback law  is such an application only depends on $\calH$.  A $T$-\textit{periodic} feedback law is a time-varying feedback law that is periodic with respect to time, $i.e.$ $U(T+t; y)= U(t; y)$. \\
A \emph{proper} feedback law $U$, roughly speaking,  is some time-varying feedback law such that,    for every $s\in \R$,  and for every $y_0 \in \calH$ as initial state at time $s$, $i. e.$ $ y(s, x)= y_0(x)$, the  closed-loop system \eqref{cons1}--\eqref{cons2} admits a unique solution.  For the closed-loop system with  proper feedback law  we can  define the ``\textit{flow}", $\Phi(t, s; y_0)$, as the state  at time  $t$ of the solution of \eqref{cons1}--\eqref{cons2}   with initial state $y(s, x)= y_0(x)$. \\

The local controllability and stabilization of Navier--Stokes equations have been extensively studied in the literature.  Based on global Carleman estimates introduced by Fursikov--Imanuvilov \cite{Fursikov-Imanuvilov-book-1997}  a nearly complete local exact controllability result is obtained in \cite{FGIP-2004}, other  works include but not limited to \cite{ Coron-Guerrero, Fursikov-1996, Ervedoza-Glass-Guerrero, Lions-Zuazua, Schmidt-Trelat}.   Eventually one can even  control the system locally via  some reduced control terms \cite{Coron-Lissy}.  The global controllability of Navier--Stokes equations with controlling on boundary (namely Lions' problem), which,  different from the cases on Riemannian manifolds \cite{A-S, 1996-Coron-Fursikov-RJMP},  is far away from been answered due to boundary layer difficulties, by far the best results are given by \cite{Coron-Marbach-Sueur, Coron-Marbach-Sueur-Zhang}.

The study on  local exponential stabilization around 0 and other  trajectories of Navier--Stokes equations both with  internal controls or with boundary controls is fruitful, notably based on Riccati type methods by optimal control theory.  For example, \cite{Badra-Takahashi, Barbu-Triggiani} for local exponential stabilization with finite dimensional  internal control (feedback laws);  exponential stabilization by boundary feedback laws \cite{Barbu-MAMS, Fursikov-stab, Raymond-2006}; stabilization around trajectories or  unstable steady states \cite{Barbu-Shirikyan, Kunisch-NS, Raymond-2006-2}, $etc$.  To the best of our knowledge, result on quantitative rapid stabilization or even finite time stabilization of Navier--Stokes equations is extremely limited, we refer to \cite{Coron-Xiang-2018} for a detailed review on these questions.\\

Recently, the author has introduced a method to stabilize the multi-dimensional  heat equations in finite time \cite{Xiang-heat-2020},  which is based on  quantitative rapid stabilization  relying on   \textit{spectral estimates}  and \textit{Lyapunov functionals}, as well as  \textit{piecewise feedback laws}.  Methodologically speaking, the technical spectral estimate is achieved via local Carleman estimates on elliptic operators  up to boundaries (as always fulfilling Hörmander's pseudoconvex condition), starting from the seminal paper \cite{Lebeau-Robbiano-CPDE} these results can be regarded as standard, at least compared to wave type operators;  the Lyapunov functions \cite{Coron-trelat-2004}  aim at finding artfully chosen energy and  multiplier to characterize the variation of the energy from a global point of view without knowing any microlocal information, which have been extensively developed  in the study of hyperbolic systems of conservation laws as well as scattering theory \cite{Bastin-Coron-book, Hayat-2019, Lions-Hilbert, ZhangRapidStab}; the piecewise  (in time)  feedback law is introduced in \cite{2017-Coron-Nguyen-ARMA} to stabilize the one dimensional heat equation in finite time together with the backstepping method,  instead of using stationary feedback laws.  This method  shares several  advantages: 
\begin{itemize}
\item  The designed feedback laws are simple and explicit to be compared with some other stabilization techniques, for instance the powerful Riccati method requires on solving some algebraic nonlinear Riccati equation;
\item  The quantitative rapid stabilization combined with the piecewise continuous feedback law argument  leads to  null controllability  without applying  Lions'  fundamental H.U.M. \cite{Lions-Hilbert}. Moreover, this  constructive approach also  provides explicit (and probably optimal) controlling costs; 
\item  The  feedback laws are stable under perturbation.   As a direct consequence,  the same feedback law can be used to stabilize (rapidly or even in finite time) nonlinear models with satisfying costs.\\
\end{itemize}

Inspired by   \cite{Xiang-heat-2020} we prove the following theorems concerning  quantitative rapid stabilization, local null controllability with cost estimates, and finite time stabilization for  the  two dimensional incompressible  internal controlled  Navier--Stokes equations, the proofs of which are presented in Section \ref{sec-rap}, Section \ref{sec-null}, and Section \ref{sec-finite} respectively.

\begin{theorem}[Quantitative rapid stabilization]\label{int-thm-rap-sta-li}
There exists an effectively computable constant $C_2>0$ such that for any $\lambda> 0$ we can construct an explicit stationary feedback law $\mathcal{F}_{\lambda}: \calH \rightarrow \calH$,  such that   the closed-loop system \eqref{cons1}--\eqref{cons-f}
with the feedback law  $U(t;y):= \mathcal{F}_{\lambda} y$ is locally exponentially stable:
\begin{align}
||\Phi(t, s; y_0)||_{\calH}+ ||\mathcal{F}_{\lambda}\Phi(t, s; y_0)||_{\calH}&\leq 2C_2 e^{C_2\sqrt{\lambda}} e^{-\frac{\lambda}{4} (t-s)}||y_0||_{\calH}, \; \forall \; s\in \R,  \forall  \;t\in [s, +\infty), \notag
\end{align}
for any $ ||y_0||_{\calH}\leq C_2^{-1} e^{-C_2 \sqrt{\lambda}}$.
\end{theorem}

\begin{theorem}[Quantitative null controllability with cost estimates]\label{int-thm-null-col}
There exists an effectively computable constant $C_3>0$ such that, for any $T\in (0, 1)$, and  for any $||y_0||_{\calH}\leq e^{-\frac{C_3}{T}}$ we can find an explicit control $f|_{[0, T]}(t, x)$ satisfying 
\begin{equation*}
||  f(t, x)||_{L^{\infty}(0, T; \calH)}\leq   e^{\frac{C_3}{T}} ||y_0||_{\calH},
\end{equation*}
such that the unique solution of  the controlled system \eqref{cons1}  with initial state $y(0, x)= y_0(x)$ and the control $f|_{[0, T]}$  verifies $ y(T, x)=0$.
\end{theorem}

\begin{theorem}[Small-time local stabilization with explicit feedback laws]\label{int-thm-semi-stab}
For any $T>0$, we find an effectively computable constant $\Lambda_{T}$ and construct an explicit  $T$-periodic proper feedback law $U$ satisfying 
\begin{equation*}
|| U(t; y)||_{\mathcal{H}}\leq \min \{1,  2||y||_{\calH}^{1/2}\}, \; \forall  \;y\in\calH, \; \forall \; t\in \R,
\end{equation*}
 that stabilizes system \eqref{cons1}--\eqref{cons2} in finite time:
\begin{itemize}
\item[(i)] ($2T$ stabilization) $\Phi(2T+ t, t; y_0)=0, \;\;\forall \;t\in \mathbb{R},\; \forall\; ||y_0||_{\calH}\leq \Lambda_{T}$.
\item[(ii)] (Uniform stability)
For every  $\delta> 0$ there exists an effectively computable $\eta> 0$ such that
\begin{equation*}
\big(||(y_0||_{\calH}\leq \eta\big) \Rightarrow \left(||\Phi(t, t'; y_0)||_{\calH}\leq \delta, \;\forall \;  t'\in \R,\; \forall\; t\in ( t', +\infty) \right).
\end{equation*}
\end{itemize}
\end{theorem}

\section{Preliminary}
\subsection{Functional  framework}\label{sec-spec}
We refer to the book by Chemin \cite{Chemin-book} for the functional analysis framework and well-posedness results concerning Navier-Stokes equations, and the book by Coron \cite{Coron} for introduction on the  related control theory. In the context if there is no confusing sometimes we simply denote $L^2(\Omega)^2$ by $L^2(\Omega)$ or $L^2$.  

\noindent (1) \textit{Leray projection and spectral decomposition.}

According to Helmholtz decomposition, for any $u\in L^2(\Omega)^2$ there exist unique $v\in \calH$ and $\nabla p\in L^2(\Omega)^2$ such that $u= v+ \nabla p$,  which defines the (orthogonal) Leray projection $\mathbb{P}$ on $L^2(\Omega)^2$:
\begin{equation*}
\left\{\begin{array}{cccc}
\mathbb{P}: & L^2(\Omega)^2 &\to &   \calH
\\
&u&\mapsto & \mathbb{P}u:= u-\nabla p.
\end{array}
\right.
\end{equation*}
Notice that for any $f\in \calH$,
\begin{equation*}
||\mathbb{P}\big(1_{\omega}f\big)||_{\calH}\leq ||1_{\omega}f||_{L^2(\Omega)^2}\leq   ||f||_{L^2(\Omega)^2}=||f||_{\calH},
\end{equation*}
which allows us to estimate the control term via $||f||_{L^2}$ (or equivalently $||f||_{\calH}$).

Let $\{e_i\}_{i=1}^{\infty}\subset \calVs$ be the orthonormal basis of $\calH$ given by the eigenvectors of the 
 the Stokes operator 
 \begin{gather*}
\begin{cases}
 -\Delta e_i+ \nabla p_i  = \tau_i e_i & \textrm{ in } \Omega,  \\
\textrm{div } e_i= 0  & \textrm{ in } \Omega, \\
e_i=0 &\textrm{ on } \partial \Omega,
\end{cases}
\end{gather*} 
with $ 0<\tau_1\leq \tau_2\leq \tau_3\leq...\leq \tau_n\leq...$ and $\lim_{i\rightarrow \infty} \tau_i= +\infty$.   Let $\calH_N$ be the low frequency subspace  of $\calH$, and $\mathbb{P}_N$ be its orthogonal projection,
\begin{equation*}
\calH_N:= \textrm{Vect} \{e_i\}_{i=1}^N.
\end{equation*}
In terms of the above eigenvectors  Leray projection can be extended to $\calV'$,
\begin{equation*}
\left\{\begin{array}{cccc}
\mathbb{P}: & \calV' &\to &   \calV'
\\
&u&\mapsto & \mathbb{P} u:= u-\nabla p,
\end{array}
\right.
\end{equation*}
where  $p\in L^2_{loc}(\Omega)$,  and $\nabla p$ belongs to   $\mathcal{V}^0_{\sigma}$ as polar space of $\calVs$,
\begin{equation*}
 \mathcal{V}^0_{\sigma}:= \left\{f\in \calV'  :  \left\langle f, v\right\rangle_{\calV'\times \calV}=0, \;\forall\; v\in \calVs\right\}.
\end{equation*}
More precisely, 
\begin{align*}
\mathbb{P} u&:=  \sum_{i=1}^{\infty} \big\langle u, e_i\big\rangle_{\calV'\times\calV}  e_i \in \calV' \; \;\;  \textrm{ for } \; u\in \calV', \\
 \mathbb{P}_N u&:=  \sum_{i=1}^N \big\langle u, e_i\big\rangle_{\calV'\times\calV}  e_i\in \calVs\; \;\;   \textrm{ for } \; u\in \calV', \\
\mathbb{P} u&:=  \sum_{i=1}^{\infty} \big(u, e_i\big)_{L^2(\Omega)^2} e_i \in \calH \; \;\;  \textrm{ for } \; u\in L^2(\Omega)^2,
\end{align*}
with
\begin{equation*}
\big\langle u, v\big\rangle_{\calV'\times\calV}= \big\langle u, v\big\rangle_{\calVa\times\calVs}=  \big\langle \mathbb{P}u, v\big\rangle_{\calVa\times\calVs},   \;\forall \; u\in \calV', \;\forall \; v\in \calVs.
\end{equation*}
Furthermore,  the related $\calH$-norm,   $\calV$-norm, and   $\calVa$-norm  can be   characterized by  
\begin{align*}
||u||_{\calH}^2&= \sum_{i=1}^{\infty}  \big|\big(u, e_i\big)_{L^2(\Omega)^2}\big|^2 \;\; \;  \textrm{ for } u\in \calH, \\
||u||_{\calV}^2&= \sum_{i=1}^{\infty}  \big|\big(u, e_i\big)_{L^2(\Omega)^2}\big|^2\tau_i \; \;\;  \textrm{ for } u\in \calV, \\
||u||_{\calVa}^2&= ||\mathbb{P}u||_{\calVa}^2= \sum_{i=1}^{\infty}  \big|\big\langle u, e_i\big\rangle_{\calV\times\calV'}\big|^2\tau_i^{-1} \; \; \;  \textrm{ for } u\in \calV'.
\end{align*}

\noindent (2) \textit{Spectral estimates.}

For any  $\lambda>0$, we denote by $N(\lambda)$  the number of the  eigenvalues   that are smaller than or equal to  $\lambda$, $i. e. $ $ \tau_{N(\lambda)}\leq \lambda<\tau_{N(\lambda)+1}$,    
and define the following  symmetric matrix $J_{N(\lambda)}$,
\begin{equation}\label{def-JN}
J_{N(\lambda)}:= \left( \big(e_i, e_j\big)_{L^2(\omega)^2} \right)_{i, j=1}^{N(\lambda)}.   
\end{equation}

\begin{proposition}[Spectral estimates]\label{Key-Lemma} 
There exists an effectively computable constant $C_1\geq 1$ only depends on $(\Omega, \omega)$ that is independent of $\lambda> 0$ such that,  for any $\lambda>0$ and for   any $E_{N(\lambda)}=(a_1, a_2, ..., a_{N(\lambda)})\in \R^{N(\lambda)}$  the following inequality holds,
\begin{equation*}
E_{N(\lambda)}^{T}J_{N(\lambda)} E_{N(\lambda)}\geq C_1^{-1} e^{-C_1\sqrt{\lambda}} ||E_{N(\lambda)}||_2^2.
\end{equation*}
\end{proposition}
\begin{proof}
This is a Lebeau--Robbiano type spectral inequality  \cite{Lebeau-Robbiano-CPDE}.
By letting $N$ be presenting $N_{\lambda}$,  we get
\begin{equation*}
E_N^{T}J_N E_N= \sum_{1\leq i, j\leq N} a_i \big(e_i, e_j\big)_{L^2(\omega)^2} a_j= ||\sum_{i=1}^N a_i e_i||_{L^2(\omega)^2}^2\geq C_1^{-1} e^{-C_1\sqrt{\lambda}}||E_N||_2^2,
\end{equation*}
where we have recalled  the recent paper \cite[Theorem 3.1]{CS-Lebeau-2016},
\[ C_1e^{C_1\sqrt{\lambda}} \int_{\omega}\left( \sum_{\tau_i\leq \lambda} a_i e_i(x)\right)^2 dx \geq \sum_{\tau_i\leq \lambda} a_i^2. \]
\end{proof}
\noindent (3) \textit{Nonlinear terms.}

Next, we define the following  bilinear map $\calQ$  as well as  trilinear functional  $\mathcal{B}$,
\begin{gather*}
\left\{\begin{array}{cccc}
\mathcal{Q}:& \calV\times \calV &\to &   \calV'
\\
& (u, v)&\mapsto & -\textrm{div }(u\otimes v)
\end{array}, \label{Q} 
\right.\\
\mathcal{B}(u, v, w):= \big\langle \mathcal{Q}(u, v), w\big\rangle_{\calV'\times \calV}= \int_{\Omega} [(u\cdot \nabla) v]\cdot w dx, \;\forall \; u, v, w \in \calV.
\end{gather*}

\begin{proposition}[Nonlinearity estimates]\label{non-es-Q}
There exists a constant $c_0$ such that for any $u, v$ and $w$ in $\calV$, we have the following estimates,
\begin{gather*}
\mathcal{B}(u, v, w)+ \mathcal{B}(u, w, v)=0 \;\textrm{ if } u\in \calVs, \\
|\mathcal{B}(u, v, w)|\leq c_0 ||u||_{L^2}^{\frac{1}{2}} ||v||_{L^2}^{\frac{1}{2}} ||\nabla u||_{L^2}^{\frac{1}{2}} ||\nabla v||_{L^2}^{\frac{1}{2}}||\nabla w||_{L^2}.
\end{gather*}
\end{proposition}

\subsection{Open loop controlled (inhomogenous) Navier-Stokes systems}
The open loop controlled equation is indeed an inhomogenous equation with a force term  located in the controlled domain.  A general  inhomogenous equation  (without any restriction on force terms) is presented by,
\begin{equation}\label{Cauch-NS}
\begin{cases}
y_t- \Delta y+ \left(y\cdot \nabla\right) y+ \nabla p=  f(t, x), &  (t, x)\in (t_1, t_2)\times \Omega,\\
\textrm{div }y(t, x)=0, & (t, x)\in (t_1, t_2)\times \Omega,\\
y(t, x)=0, & (t, x)\in (t_1, t_2)\times \partial \Omega,\\
y(t_1, x)= y_0(x),  & x\in  \Omega,
\end{cases}
\end{equation} 
where $t_2$ can be taken as $+\infty$.  We are interested in   the solutions under  Leray's weak solution sense \cite{Leray}: for any $y_0\in \calH$ and any $f\in L^2_{loc}(t_1, t_2; \calV')$, the solution of equation \eqref{Cauch-NS} is some $y\in C([t_1, t_2]; \calH)\cap L^2_{loc}(t_1, t_2; \calVs)$ such that, for any test function $\phi$ in $C^1([t_1, t_2]; \calVs)$, the vector field $y$ satisfies the following condition: 
\begin{gather}
\big( y(t), \phi(t)\big)_{\calH}=  \big( y_0, \phi(0)\big)_{\calH}+ \int_{t_1}^t\big\langle \Delta \phi(s)+ \partial_t\phi(s), y(s)\big\rangle_{\calVa\times\calVs} ds \;\;\;\;\;\;\;\;\;\;\;\;\;\;\;\;\;\;\;\;\;\;\;\;\;\notag \\
\;\;\;\;\;\;\;\;\;\;\;\;\;\;\;\;\;\;\;\;\;\;\;\;\;\;\;\;\;\;\;\;\;\;+ \int_{t_1}^t \big(y(s)\otimes y(s), \nabla \phi(s)\big)_{L^2(\Omega)^2} ds+ \int_{t_1}^t \big\langle f(s), \phi(s)\big\rangle_{\calVa\times \calVs} ds,
\end{gather}
for every $ t\in [t_1, t_2]$.

\begin{theorem}[Leray theorem on well-posedness and stability of the solutions]\label{Thm-well-stab}
For any $y_0\in \calH$ and any $f\in L^2_{loc}(t_1, t_2; L^2(\Omega)^2)$, the Cauchy problem \eqref{Cauch-NS} admits a unique solution. This unique solution is also in $H^1_{loc}(t_1, t_2; \calV')$. Moreover,  there exists some constant $C_0$ independent of $t_1$ and $t_2$ such that  this unique solution satisfies, 
\begin{align}
\frac{1}{2}||y(t, x)||_{\calH}^2+ \int_{t_1}^t||\nabla y(s, x)||_{L^2}^2 ds&= \frac{1}{2}||y_0||_{\calH}^2+\int_{t_1}^{t}\big\langle f(s), y(s)\big\rangle_{\calVa\times \calVs} ds, \\
||y(t, x)||_{\calH}^2+ \int_{t_1}^t||\nabla y(s, x)||_{L^2}^2 ds&\leq ||y_0||_{\calH}^2+ C_0 \int_{t_1}^t||f(s)||_{L^2}^2 ds,
\end{align}
for any $t\in [t_1, t_2]$.

Furthermore, the Leray solutions are stable in the following sense. Let $y$ (resp. $z$) be the Leray solution associated with $y_0$ (resp. $z_0$) in $\calH$ and $f$ (resp. $g$) in the space $L^2_{loc}(t_1, +\infty; L^2(\Omega)^2)$, then for $w:= y-z$ and for  any $t\in (t_1, +\infty)$ we have 
\begin{gather*}
||w(t)||_{\calH}^2+  \int_{t_1}^t||\nabla w(s, x)||_{L^2}^2 ds\leq \left(||w_0||_{\calH}^2+ C_0 \int_{t_1}^t||(f-g)(s)||_{L^2}^2 ds\right) \exp (C E^2(t)), \\
E(t):= \min \left\{||y_0||_{\calH}^2+  C_0\int_{t_1}^t ||f(s)||_{L^2}^2 ds,  ||z_0||_{\calH}^2+ C_0\int_{t_1}^t ||g(s)||_{L^2}^2 ds \right\}.
\end{gather*}
\end{theorem}
\noindent Actually Theorem \ref{Thm-well-stab} holds for $f$ in  $L^2_{loc}(t_1, +\infty; \calV')$, for which  the related inequalities are governed  by the  $L^2(\calVa)$-norm of $f$ and  the constant $C_0$ can be taken as 1.  

\subsection{Time-varying feedback laws, closed-loop systems, and finite time stabilization}\label{sec-def-prop-feed}
In this section we recall the precise definition of  time-varying feedback laws as well as the related closed-loop solutions.
\begin{defi}[Closed-loop systems]
\label{def-sol-closed-loop}
Let $s_1\in \R$ and $s_2\in \R$ be given such that $s_1<s_2$. Let the time-varying feedback law on interval $[s_1, s_2]$  be an application
\begin{equation}\label{feed-s1s2}
\left\{ 
\begin{array}{cccc}
U:  & [s_1,s_2]\times\calH &\to & \calH
\\
&(t; y)&\mapsto & U(t; y).
\end{array}
\right.
\end{equation}
Let $t_1\in [s_1,s_2]$, $t_2\in (t_1, s_2]$, and
$y_0\in \calH$. A solution on $[t_1,t_2]$ to the Cauchy problem associated to 
the closed-loop system \eqref{cons1}--\eqref{cons-f} with \eqref{feed-s1s2} for initial data $y_0$ at time $t_1$ is some $y:  [t_1,t_2]\to \calH$ such that
\begin{gather*}
t\in (t_1,t_2)\mapsto f(t, x):= 1_{\omega}U(t; y(t))\in L^2(t_1, t_2; L^2(\Omega)^2),
\\
y \text{ is a Leray solution  of \eqref{Cauch-NS} with initial data $y_0$ at time $t_1$ and the above force term $f(t, x)$.}
\end{gather*}
\end{defi}
\begin{defi}[Proper feedback laws]
\label{defproper}
Let $s_1\in \R$ and $s_2\in \R$ be given such that $s_1<s_2$. A proper feedback law on $[s_1,s_2]$ is an application $U$ of type  \eqref{feed-s1s2} 
such that, for every $t_1\in [s_1,s_2]$, for every $t_2\in (t_1, s_2]$, and for every
$y_0\in \calH$,  there exists a unique solution on $[t_1,t_2]$ to the Cauchy problem associated to 
the closed-loop system \eqref{cons1}--\eqref{cons-f} with \eqref{feed-s1s2} for initial data $y_0$ at time $t_1$ according to  Definition~\ref{def-sol-closed-loop}.

A proper feedback law  is an application $U$ of type \eqref{cons2}
such that, for every $s_1\in \R$ and for every $s_2\in \R$ satisfying $s_1<s_2$,
the feedback law restricted to $[s_1,s_2]\times \calH$ is a proper feedback law on $[s_1,s_2]$.
\end{defi}

\noindent For a \textbf{proper} feedback law,    one can define the \textbf{flow} $\Phi: \Delta \times \calH \to \calH$ associated to this feedback law, with $\Delta:=\{(t,s);\, t> s\}$:   $\Phi(t,s; y_0)$ is the value at time $t$ of the solution $y$ to the closed-loop system \eqref{cons1}--\eqref{cons2} which is equal to $y_0$ at time $s$.

\begin{defi}[Finite time local stabilization of Navier-Stokes equations]\label{def-fi-sta}
Let $T> 0$.  A $T$-periodic proper feedback law $U$ locally stabilizes the two dimensional Navier-Stokes equation in finite time, if  for some $\varepsilon>0$ the flow $\Phi$ of  the closed-loop system \eqref{cons1}--\eqref{cons2} verifies, 
\begin{itemize}
\item[(i)] ($2T$ stabilization) $\Phi(2T+ t, t; y_0)=0, \;\;\forall t\in \mathbb{R},\; \forall \; ||y_0||_{\calH}\leq \varepsilon,$
\item[(ii)] (Uniform stability)
For every  $\delta> 0$, there exists $\eta> 0$ such that
\begin{equation*}
\big(||y_0||_{\calH}\leq \eta\big) \Rightarrow \left(||\Phi(t, t'; y_0)||_{\calH}\leq \delta, \;\forall   \; t'\in \R,\; \forall  \;t\in ( t', +\infty) \right).
\end{equation*}
\end{itemize}
\end{defi}

\subsection{Well-posedness of closed-loop systems}
Finally we present  well-posedness results concerning closed-loop systems with stationary Lipschitz feedback laws.  Concerning linear feedback laws one has the following well-posedness  result.
\begin{theorem}\label{thm-clo-sta-F-li}
Let $T> 0$.   Let given vector  functions $\{\varphi_i\}_{i=1}^n\in \calH$ and  bounded linear   operators $\{l_i\}_{i=1}^n: \calH\rightarrow \mathbb{R}$.  For any  $y_0\in \calH$  the Cauchy problem
\begin{equation*}
\begin{cases}
y_t- \Delta y+ \left(y\cdot \nabla\right) y+\nabla p   = 1_{\omega} \Big(\sum_{i=1}^n l_i(y) \varphi_i\Big), &  (t, x)\in (0, T)\times \Omega,\\
\textrm{div } y= 0, &  (t, x)\in (0, T)\times \Omega,\\
y(t, x)=0, & (t, x)\in (0, T)\times \partial \Omega,\\
y(0, x)= y_0(x), & x\in \Omega,
\end{cases}
\end{equation*} 
admits a unique solution.  
\end{theorem}

As  will consider  finite time  stabilization problems we also introduce ``cutoff" type feedback laws.   For any $r\in (0, 1/2]$ we find some  smooth \textit{cutoff} function  $ \chi_r\in C^{\infty}(\mathbb{R}^+; [0, 1])$ satisfying 
\begin{equation}\label{def-chi-r}
\chi_r(x)= \begin{cases}
1  \;\;\textrm{ if } x\in [0, r], \\ 0 \;\; \textrm{ if } x\in [2r, +\infty),
\end{cases}
\end{equation}
and further define  the  following related \textit{Lipschitz} operator $\mathcal{K}_r: \calH\rightarrow \calH$,
\begin{equation}\label{def-FF}
\mathcal{K}_r (y):= y \cdot \chi_r\left(||y||_{\calH}\right), \;\forall \; y\in \calH.
\end{equation}
satisfying, for some constant $L_r$ depending on $r$,
\begin{gather*}
||\mathcal{K}_r (y)||_{\calH}\leq \min\{1, ||y||_{\calH}\},\\
||\mathcal{K}_r (y)-\mathcal{K}_r (z)||_{\calH}\leq L_r ||y-z||_{\calH}, \; \forall\; y, z\in \calH.
\end{gather*}

\begin{theorem}\label{thm-clo-sta-F}
Let $T>0$.  Let $r\in (0, 1/2]$. Let given vector  functions $\{\varphi_i\}_{i=1}^n\in \calH$ and bounded  linear   operators $\{l_i\}_{i=1}^n: \calH\rightarrow \mathbb{R}$.  For any  $y_0\in \calH$  the Cauchy problem
\begin{equation*}
\begin{cases}
y_t- \Delta y+ \left(y\cdot \nabla\right) y+\nabla p   = 1_{\omega} \mathcal{K}_r\Big(\sum_{i=1}^n l_i(y) \varphi_i \Big), &  (t, x)\in (0, T)\times \Omega,\\
\textrm{div } y= 0, &  (t, x)\in (0, T)\times \Omega,\\
y(t, x)=0, & (t, x)\in (0, T)\times \partial \Omega,\\
y(0, x)= y_0(x), &x\in \Omega,
\end{cases}
\end{equation*} 
admits a unique solution.  
\end{theorem}
Both the closed-loop systems with  linear feedback laws and   with Lipschitz nonlinear feedback laws are well-posed correspond to Theorem  \ref{thm-clo-sta-F-li} and Theorem \ref{thm-clo-sta-F}, the detailed proofs of which  we omit. Indeed, local  (in time) existence and uniqueness of solutions   are based on Leray's theorem \ref{Thm-well-stab} and Banach fixed point theorem: let the Lipschitz constant of the feedback law be $L$ and let $||y_0||_{\calH}=M$, for some $\tilde{T}$ small enough we consider the Banach space, 
\begin{gather*}
\mathcal{X}_{\tilde{T}}:= C([0, \tilde{T}]; \calH)\cap L^2(0, \tilde{T}; \calVs), \\
\mathcal{X}_{\tilde{T}}(2M):= \left\{y\in \mathcal{X}_{\tilde{T}}: ||y||_{\mathcal{X}_{\tilde{T}}}^2= ||y||_{C([0, \tilde{T}]; \calH)}^2+ ||\nabla y||_{L^2(0, \tilde{T}; L^2)}^2\leq 4M^2\right\},
\end{gather*}
and find  the fixed point of the following application, 
\begin{equation*}
\left\{ 
\begin{array}{cccc}
\mathcal{S} : & \mathcal{X}_{\tilde{T}}(2M) &\to &  \mathcal{X}_{\tilde{T}}
\\
& y &\mapsto & \mathcal{S}(y),
\end{array}
\right.
\end{equation*}
where $\mathcal{S}(y)$ is the solution of Cauchy problem \eqref{Cauch-NS} with the  initial state $y_0$ and force  (control) term $f=  1_{\omega} \mathcal{K}_r\big(\sum_{i=1}^n l_i(y) \varphi_i \big)$.  Moreover, since $y=\mathcal{S} (y)$ is the solution of the Cauchy problem \eqref{Cauch-NS} with control  $f=  1_{\omega} \mathcal{K}_r\big(\sum_{i=1}^n l_i(y) \varphi_i \big)$, thanks to Theorem \ref{Thm-well-stab},  this solution also belongs to   the space $H^1(0, T; \calV')$. 
In the end, some \textit{a priori} energy estimates lead to global (in time) solutions. \\We also emphasize that the Lipschitz condition is crucial in order to guarantee the uniqueness.  Otherwise one may  need to use other compactness arguments to prove  existence of solutions, see for example \cite{coron-rivas-xiang-kdv-16}.

\subsection{On the choice of constants}
In this section we conclude the values of the constants that will be used later on.

$\bullet$ For any given $\lambda>0$,  we define 
\begin{equation}\label{def-gam-mu}
\gamma_{\lambda}:= C_1e^{C_1\sqrt{\lambda}} \lambda \; \;\textrm{ and }\;   \; \mu_{\lambda}:= \frac{\gamma_{\lambda}^2}{\lambda^2}= C_1^2e^{2C_1\sqrt{\lambda}}>1.
\end{equation}

$\bullet$ By recalling the definition of  $C_1$ in Proposition \ref{Key-Lemma},  we  further select  some $C_2\in [3C_1, +\infty)$ such that for all $\lambda>0$,
\begin{equation}\label{def-C2}
 (1+\lambda C_1) e^{C_1 \sqrt{\lambda}}, \; 8(1+\lambda) C_1^2 e^{2C_1 \sqrt{\lambda}}, \;  8c_0C_1^3e^{3C_1\sqrt{\lambda}} \leq C_2 e^{C_2\sqrt{\lambda}},
\end{equation}
and define 
\begin{equation}
r_{\lambda}:= \left(C_2 e^{C_2\sqrt{\lambda}}\right)^{-1}.
\end{equation}

$\bullet$ Then we  choose some constant  $Q>0$  satisfying
\begin{equation}
C_1 e^{C_1 Q m}, C_2e^{C_2 Q m}\leq e^{\frac{Q^2}{64} m}, \; \forall \; m\geq 1,
\end{equation}
and select
\begin{equation}
 C_3:= \frac{Q^2}{32}.
\end{equation}

\section{Quantitative rapid stabilization}\label{sec-rap} 
Inspired by the recent work \cite{Xiang-heat-2020} on the stabilization of the heat equations, we directly  define the following  stationary feedback  law 
\begin{equation}\label{def-F-lambda}
\mathcal{F}_{\lambda} y:= -\gamma_{\lambda} \mathbb{P}_{N(\lambda)} y, \; \forall y\in \calH,
\end{equation} 
and consider the following closed-loop system,
\begin{equation}\label{clo-1} 
\begin{cases}
y_t= \Delta y- (y\cdot \nabla)y -\nabla p- \gamma_{\lambda} 1_{\omega} \mathbb{P}_{N} y &\textrm{ in } \Omega, \\
\textrm{div } y=0  &\textrm{ in } \Omega,\\
y=0 & \textrm{ on } \partial \Omega,
\end{cases}
\end{equation} 
where, and from now on, we simply denote $N(\lambda)$ by $N$.  Furthermore, the low frequency system satisfies
\begin{equation}
\frac{d}{dt} \Big(\mathbb{P}_N y\Big)=    \mathbb{P}_N \big(\Delta y\big)-  \mathbb{P}_N\Big((y\cdot \nabla)y\Big) - \gamma_{\lambda} \mathbb{P}_N \Big(1_{\omega} \mathbb{P}_{N} y\Big).
\end{equation} 
Since $y$ lives in $\calH$, we decompose
\begin{gather*}
y(t, x)=\mathbb{P}y(t, x)= \sum_{i=1}^{\infty} y_i(t) e_i,\\
\mathbb{P}\left(1_{\omega} e_i\right)=\sum_{j=1}^{\infty} \big(1_{\omega} e_i, e_j\big)_{L^2(\Omega)^2} e_j=  \sum_{j=1}^{\infty} \big(e_i, e_j\big)_{L^2(\omega)^2} e_j,
\end{gather*}
and 
\begin{gather*}
 \mathbb{P}_N \big((y\cdot \nabla)y\big)= \sum_{i=1}^N \big\langle (y\cdot \nabla)y, e_i\big\rangle_{\calV'\times \calV} e_i,\\
 \mathbb{P}_N \Big(1_{\omega} \mathbb{P}_{N} y\Big)= \mathbb{P}_N \Big(1_{\omega}  \sum_{i=1}^{N} y_i(t) e_i\Big)= \sum_{i=1}^N \sum_{j=1}^{N} y_i(t) \big(e_i, e_j\big)_{L^2(\omega)^2} e_j.
\end{gather*}
By defining 
\begin{equation}
X_N(t):= \begin{pmatrix}
y_1(t)\\ y_2(t)\\...\\y_N(t)
\end{pmatrix}, 
Y_N(t):= \begin{pmatrix}
-\big\langle (y\cdot \nabla)y, e_1\big\rangle_{\calV'\times \calV}(t)\\ -\big\langle (y\cdot \nabla)y, e_2\big\rangle_{\calV'\times \calV}(t)\\...\\-\big\langle (y\cdot \nabla)y, e_N\big\rangle_{\calV'\times \calV}(t)
\end{pmatrix}, 
A_N:= \begin{pmatrix}
-\tau_1 & & &\\ & -\tau_2 & & \\&&...&\\& & &-\tau_N
\end{pmatrix},
\end{equation}
we know that the finite dimensional system  $X_N(t)$ satisfies 
\begin{equation}
\dot{X}_N(t)= A_N X_N(t)-\gamma_{\lambda} J_N X_N(t)+ Y_N(t).
\end{equation}

Let us consider  the following  Lyapunov functional on $\calH$,
\begin{equation}\label{def-Vy}
V(y):= \mu_{\lambda}  \left(\mathbb{P}_N y, \mathbb{P}_N y\right)_{L^2(\Omega)^2}+ \left(\mathbb{P}_N^{\perp} y, \mathbb{P}_N^{\perp} y\right)_{L^2(\Omega)^2}=  \mu_{\lambda} ||X_N||_2^2+ \left(\mathbb{P}_N^{\perp} y, \mathbb{P}_N^{\perp} y\right)_{L^2(\Omega)^2}, 
\end{equation}
for every $y\in \calH$, where  $||X_N||_2^2$ denotes $\sum_{i=1}^N y_i^2$. 

Concerning the variation of the  Lyapunov functional, at least for $y(t)$ regular enough, for example $C^1([0, T]; \calVa)\cap C^0([0, T]; \calVs)$, one has
\begin{align*} 
\frac{d}{dt}V\left(y(t)\right)&= \mu_{\lambda} \frac{d}{dt} ||X_N||_2^2+ \frac{d}{dt} \left(\mathbb{P}_N^{\perp} y, \mathbb{P}_N^{\perp} y\right)_{L^2(\Omega)^2}= 2\mu_{\lambda}\dot{X}_N^T X_N+  2\left\langle\mathbb{P}_N^{\perp} y,  \frac{d}{dt}y\right\rangle_{\calVs\times \calVa},
\end{align*}
and  the value of $-Y_N^T X_N$ is given by 
\begin{equation*}
\left\langle\mathbb{P}_N \Big((y\cdot \nabla)y\Big),  \mathbb{P}_N y\right\rangle_{\calVa\times \calVs}= \left\langle\mathbb{P} \Big((y\cdot \nabla)y\Big),  \mathbb{P}_N y\right\rangle_{\calVa\times \calVs}=  \big\langle(y\cdot \nabla)y,  \mathbb{P}_N y\big\rangle_{\calV'\times \calV}= \mathcal{B}\left(y, y, \mathbb{P}_N y\right).
\end{equation*} 
Hence, on the one hand,
thanks to Proposition \ref{Key-Lemma},  Proposition \ref{non-es-Q},  and  the choice of $\gamma_{\lambda}$ and $\mu_{\lambda}$, we have
\begin{align*}
\mu_{\lambda} \frac{d}{dt} ||X_N||_2^2&= \mu_{\lambda} X_N^T \left( A_N^T+ A_N-2\gamma_{\lambda} J_N\right)X_N+ \mu_{\lambda} \left( Y_N^T X_N+ X_N^T Y_N\right)  \\
&= 2\mu_{\lambda} X_N^T \left(  A_N-\gamma_{\lambda} J_N\right)X_N+ 2\mu_{\lambda} X_N^T Y_N, \\
&\leq -2\mu_{\lambda} \gamma_{\lambda} \left(C_1e^{C_1\sqrt{\lambda}}\right)^{-1}||X_N||_2^2- 2\mu_{\lambda} ||\nabla \mathbb{P}_N y||_{L^2(\Omega)}^2+  2\mu_{\lambda} |\mathcal{B}\left(y, y, \mathbb{P}_N y\right)|, \\
&\leq -2\mu_{\lambda} \lambda||X_N||_2^2- 2\mu_{\lambda} ||\nabla \mathbb{P}_N y||_{L^2(\Omega)}^2+  2\mu_{\lambda} c_0||y||_{L^2(\Omega)}||\nabla y||_{L^2(\Omega)}||\nabla \mathbb{P}_N y||_{L^2(\Omega)}, \\
&\leq -2\mu_{\lambda} \lambda||X_N||_2^2- 2\mu_{\lambda} ||\nabla \mathbb{P}_N y||_{L^2(\Omega)}^2+  2\mu_{\lambda} c_0||y||_{L^2(\Omega)}||\nabla y||_{L^2(\Omega)}^2.
\end{align*}
On the other hand, 
\begin{align*}
&\;\;\;\;\;\frac{d}{dt} \left(\mathbb{P}_N^{\perp} y, \mathbb{P}_N^{\perp}y\right)_{L^2(\Omega)^2}, \\
&= 2\left\langle\mathbb{P}_N^{\perp} y, \Delta y-(y\cdot \nabla)y - \gamma_{\lambda} 1_{\omega} \mathbb{P}_{N} y -\nabla p\right\rangle_{\calVs\times \calVa}, \\
&= -2\left(\mathbb{P}_N^{\perp} y,   y\right)_{\calVs}- 2\gamma_{\lambda} \left(\mathbb{P}_N^{\perp} y, 1_{\omega} \mathbb{P}_{N} y\right)_{L^2(\Omega)^2}- 2  \left\langle(y\cdot \nabla)y,  \mathbb{P}_N^{\perp} y \right\rangle_{\calVa\times \calVs}, \\
&=-2\sum_{i=N+1}^{\infty}\tau_i y_i^2- 2\gamma_{\lambda} \left(\mathbb{P}_N^{\perp} y, 1_{\omega} \mathbb{P}_{N} y\right)_{L^2(\Omega)^2}- 2 \mathcal{B}\left(y, y, \mathbb{P}_N^{\perp} y\right), \\
&\leq -2\sum_{i=N+1}^{\infty}\tau_i y_i^2+2\gamma_{\lambda} ||\mathbb{P}_N^{\perp} y||_{L^2(\Omega)} ||1_{\omega} \mathbb{P}_{N} y||_{L^2(\Omega)}+ 2 \mathcal{B}\left(y, y, \mathbb{P}_N y\right), \\
&\leq -\frac{3}{2}\lambda ||\mathbb{P}_N^{\perp}y||_{L^2(\Omega)}^2- \frac{1}{2} ||\nabla \mathbb{P}_N^{\perp} y||_{L^2(\Omega)}^2+ \lambda ||\mathbb{P}_N^{\perp} y||_{L^2(\Omega)}^2+ \frac{\gamma_{\lambda}^2}{\lambda} ||X_N||_2^2+   2  c_0||y||_{L^2(\Omega)}||\nabla y||_{L^2(\Omega)}^2,   \\ 
&\leq -\frac{1}{2}\lambda ||\mathbb{P}_N^{\perp}y||_{L^2(\Omega)}^2- \frac{1}{2} ||\nabla \mathbb{P}_N^{\perp} y||_{L^2(\Omega)}^2+ \mu_{\lambda}\lambda||X_N||_2^2+   2 c_0||y||_{L^2(\Omega)}||\nabla y||_{L^2(\Omega)}^2.
\end{align*}
Combining the preceding three inequalities, we further get 
\begin{align*}
\frac{d}{dt}V(y(t))&\leq -2\mu_{\lambda} \lambda||X_N||_2^2- 2\mu_{\lambda} ||\nabla \mathbb{P}_N y||_{L^2(\Omega)}^2+  2\mu_{\lambda} c_0||y||_{L^2(\Omega)}||\nabla y||_{L^2(\Omega)}^2    \\
&\;\;\;\;\;\;\;\;\;\;  -\frac{1}{2}\lambda ||\mathbb{P}_N^{\perp}y||_{L^2(\Omega)}^2- \frac{1}{2} ||\nabla \mathbb{P}_N^{\perp} y||_{L^2(\Omega)}^2+ \mu_{\lambda}\lambda||X_N||_2^2+   2 c_0||y||_{L^2(\Omega)}||\nabla y||_{L^2(\Omega)}^2, \\
&\leq -\mu_{\lambda} \lambda||X_N||_2^2- \frac{1}{2}\lambda ||\mathbb{P}_N^{\perp}y||_{L^2(\Omega)}^2- \frac{1}{2} ||\nabla  y||_{L^2(\Omega)}^2+ 4\mu_{\lambda} c_0||y||_{L^2(\Omega)}||\nabla y||_{L^2(\Omega)}^2, \\
&\leq \left(-\frac{\lambda}{2}\right)V(y(t))-\frac{1}{2} ||\nabla  y||_{L^2(\Omega)}^2+ 4\mu_{\lambda} c_0 ||y||_{L^2(\Omega)}||\nabla y||_{L^2(\Omega)}^2, \\
&\leq  \left(-\frac{\lambda}{2}\right)V(y(t))- ||\nabla y||_{L^2(\Omega)}^2\left(\frac{1}{2}- 4\mu_{\lambda} c_0 V^{\frac{1}{2}}(y(t))\right).
\end{align*}
Eventually,  according to Theorem \ref{Thm-well-stab} and Theorem \ref{thm-clo-sta-F-li} the solution $y$ lives indeed in the space  $H^1(0, T; \calV')\cap C^0([0, T]; \calH)\cap L^2(0, T; \calVs)$, which is included in $H^1(0, T; \calVa)\cap L^2(0, T; \calVs)$. Hence the preceding inequality holds in the distribution sense in $L^1(0, T)$.  Inspired by the same formula,  at first by ignoring  the first term in the right hand side  we know that the value of $\max\{V(y(t)), (8\mu_{\lambda}c_0)^{-2}\}$ decreases. Therefore, if $V(y(0))\leq (8\mu_{\lambda}c_0)^{-2}$ then  the value of  $V(y(t))$ is always smaller than $ (8\mu_{\lambda}c_0)^{-2}$. As a consequence in the preceding inequality one can  next ignore the second term involving $\nabla y$, which results in that the  Lyapunov functional $V(y(t))$  decay exponentially with decay rate $\lambda/2$. More precisely, by the choice of $r_{\lambda}$  for any initial data $y_0\in \calH$ satisfying  $||y_0||_{L^2(\Omega)}\leq r_{\lambda}$ we have
\[ V(y_0)\leq \mu_{\lambda} ||y_0||_{L^2(\Omega)}^2\leq \mu_{\lambda} r_{\lambda}^2\leq  (8\mu_{\lambda}c_0)^{-2},\]
thus
\[ V(y(t))\leq e^{-\frac{\lambda}{2} t} V(y(0)), \; \forall \; t\geq 0.\]
Consequently
\begin{equation}
||y(t)||_{L^2(\Omega)}^2\leq V(y(t))\leq e^{-\frac{\lambda}{2} t} V(y(0))\leq e^{-\frac{\lambda}{2} t}\mu_{\lambda} ||y(0)||_{L^2(\Omega)}^2\leq C_1^2 e^{2C_1\sqrt{\lambda}} e^{-\frac{\lambda}{2} t}||y(0)||_{L^2(\Omega)}^2.  \notag
\end{equation}
Therefore,  for any $ ||y_0||_{L^2(\Omega)}\leq r_{\lambda}$ the solution decays exponentially,
\begin{align*}
||y(t)||_{L^2(\Omega)}&\leq C_1 e^{C_1\sqrt{\lambda}} e^{-\frac{\lambda}{4} t}||y(0)||_{L^2(\Omega)},   \forall \; t\in [0, +\infty), \\
|| \mathcal{F}_{\lambda} y(t)||_{L^2(\Omega)}&\leq \gamma_{\lambda} ||y(t)||_{L^2(\Omega)}\leq \lambda C_1^2 e^{2C_1 \sqrt{\lambda}} e^{-\frac{\lambda}{4} t} ||y(0)||_{L^2(\Omega)},  \forall \; t\in [0, +\infty). 
\end{align*}

Therefore, we have proved the following theorem that quantifies Theorem \ref{int-thm-rap-sta-li}.
\begin{theorem}[Local stabilization with linear feedback laws]\label{thm-rap-sta-li}
For any $\lambda> 0$,   for any $||y_0||_{\calH}\leq r_{\lambda}$, and for any $s\in \mathbb{R}$ the Cauchy problem 
\begin{gather}\label{sys-li-f-lam}
\begin{cases}
y_t- \Delta y+ (y\cdot \nabla)y+ \nabla p=-\gamma_{\lambda} 1_{\omega} \mathcal{F}_{\lambda}y, & (t, x)\in [s, +\infty)\times \Omega,\\
\textrm{div } y=0, &  (t, x)\in [s, +\infty)\times \Omega,\\
y(t, x)=0, &    (t, x)\in [s, +\infty)\times \partial \Omega, \\
y(s, x)= y_0(x),& x\in \Omega,
\end{cases}
\end{gather} 
 has a unique solution in $C^0([s, +\infty); \calH)\cap L^2_{loc}(s, +\infty; \calVs)$. Moreover, this unique solution verifies 
\begin{align}
||y(t)||_{L^2(\Omega)}&\leq C_1 e^{C_1\sqrt{\lambda}} e^{-\frac{\lambda}{4} (t-s)}||y_0||_{L^2(\Omega)}, \; \forall  \;t\in [s, +\infty), \label{ex-es-1} \\
||\mathcal{F}_{\lambda}y(t)||_{L^2(\Omega)}&\leq  C_2 e^{C_2\sqrt{\lambda}} e^{-\frac{\lambda}{4} (t-s)}||y_0||_{L^2(\Omega)}, \; \forall \; t\in [s, +\infty).    \label{ex-es-2}
\end{align}
\end{theorem}
Actually similar result also holds for nonlinear feedback laws  $\mathcal{K}_{r_{\lambda}}\left( \mathcal{F}_{\lambda} y\right)$ provided by equations \eqref{def-chi-r}--\eqref{def-FF} and \eqref{def-F-lambda}.  From the preceding theorem we observe  that  for initial state  $||y_0||_{\calH}\leq r_{\lambda}^2$, the unique  solution $y(t)$  of the Cauchy problem \eqref{sys-li-f-lam} satisfies
\begin{equation}\label{li-con-r}
||\mathcal{F}_{\lambda}y(t)||_{L^2(\Omega)}\leq  C_2 e^{C_2\sqrt{\lambda}} e^{-\frac{\lambda}{4} (t-s)}||y_0||_{L^2(\Omega)}\leq  C_2 e^{C_2\sqrt{\lambda}}  r_{\lambda}^2= r_{\lambda}, \; \forall \; t\in [s, +\infty).   
\end{equation}

Now, we replace the linear  feedback law $\mathcal{F}_{\lambda}$ by  $\mathcal{K}_{r_{\lambda}}\left( \mathcal{F}_{\lambda} y\right)$ (see equation \eqref{def-FF}), which satisfies 
 \begin{equation}\label{cond-K-F}
 ||\mathcal{K}_{r_{\lambda}}\left( \mathcal{F}_{\lambda} y\right)||_{L^2(\Omega)}\leq \min \{1, \sqrt{2||y||_{L^2(\Omega)}}\}.
 \end{equation}
Indeed,  if $||\mathcal{F}_{\lambda} y||_{L^2(\Omega)}\leq 2 r_{\lambda}$, then since the operator norm $||\mathcal{F}_{\lambda}||\leq \gamma_{\lambda}\leq r_{\lambda}^{-1}$, we have,
 \begin{equation*}
 ||\mathcal{K}_{r_{\lambda}}\left( \mathcal{F}_{\lambda} y\right)||_{L^2(\Omega)}\leq || \mathcal{F}_{\lambda_n} y||_{L^2(\Omega)}\leq \sqrt{2 r_{\lambda} || \mathcal{F}_{\lambda} y||_{L^2(\Omega)}}\leq \sqrt{2 r_{\lambda} || \mathcal{F}_{\lambda}|| ||y||_{L^2(\Omega)}} \leq \sqrt{2||y||_{L^2(\Omega)}}.
 \end{equation*}
  If $||\mathcal{F}_{\lambda} y||_{L^2(\Omega)}> 2 r_{\lambda}$, then by the definition of $\mathcal{K}_{r_{\lambda}}$ we know that  $\mathcal{K}_{r_{\lambda}}\left( \mathcal{F}_{\lambda} y\right)= 0$, which completes the proof of the condition \eqref{cond-K-F}.
  
Finally, we show that for  $||y_0||_{\calH}\leq r_{\lambda}^2$  the solution of the closed-loop system with feedback law $\mathcal{K}_{r_{\lambda}}\left( \mathcal{F}_{\lambda} y\right)$  also decays exponentially.  Indeed, it suffices to prove that the solution $y$ verifies
\begin{equation*}
\mathcal{K}_{r_{\lambda}}\left( \mathcal{F}_{\lambda} y(t)\right)= \mathcal{F}_{\lambda} y(t), \; \forall \; t\in [s, +\infty),
\end{equation*}
which, by recalling the definition of $\mathcal{K}_{r_{\lambda}}$ in \eqref{def-chi-r}--\eqref{def-FF},   is true  according to \eqref{li-con-r},
\begin{equation*}
||\mathcal{F}_{\lambda}y(t)||_{L^2(\Omega)}\leq  C_2 e^{C_2\sqrt{\lambda}}||y_0||_{L^2(\Omega)}\leq C_2 e^{C_2\sqrt{\lambda}}r_{\lambda}^2= r_{\lambda}, \; \forall \; t\in [s, +\infty). 
\end{equation*}
 \begin{theorem}[Local stabilization with nonlinear Lipschitz feedback laws]\label{thm-rap-sta-li-K}
For any $\lambda> 0$,   for any $||y_0||_{\calH}\leq r_{\lambda}^2$, and for any $s\in \mathbb{R}$ the Cauchy problem 
\begin{gather}
\begin{cases}
y_t- \Delta y+ (y\cdot \nabla)y+ \nabla p=-\gamma_{\lambda} 1_{\omega}\mathcal{K}_{r_{\lambda}}\left( \mathcal{F}_{\lambda}y\right), &  (t, x)\in [s, +\infty)\times \Omega,\\
\textrm{div } y=0, &  (t, x)\in [s, +\infty)\times \Omega,\\
y(t, x)=0, &    (t, x)\in [s, +\infty)\times \partial \Omega, \\
y(s, x)= y_0(x), & x \in \Omega,
\end{cases}
\end{gather} 
 has a unique solution in $C^0([s, +\infty); \calH)\cap L^2_{loc}(s, +\infty; \calVs)$. Moreover, this unique solution verifies 
\begin{align*}
||y(t)||_{L^2(\Omega)}&\leq C_1 e^{C_1\sqrt{\lambda}} e^{-\frac{\lambda}{4} (t-s)}||y_0||_{L^2(\Omega)}, \; \forall  \;t\in [s, +\infty), \\
||\mathcal{F}_{\lambda}y(t)||_{L^2(\Omega)}&\leq  C_2 e^{C_2\sqrt{\lambda}} e^{-\frac{\lambda}{4} (t-s)}||y_0||_{L^2(\Omega)}, \; \forall \; t\in [s, +\infty).   
\end{align*}
\end{theorem}

\section{Quantitative null controllability with cost estimates}\label{sec-null}
In this section we  construct feedback laws (controls) that yields the solution decays to zero in finite time. 
\begin{theorem}\label{thm-null-col-opt}
There exists $C_3>0$ such that, for any $T\in (0, 1)$,  for any $y_0\in \calH$ satisfying $||y_0||_{\calH}\leq e^{-\frac{C_3}{T}} $  we construct  an explicit control $f(t, x)$ for the controlled system \eqref{cons1}  such that the unique solution with initial data $y(0, x)=y_0(x)$ verifies $y(T, x)=0$.  Moreover, 
\begin{equation}
|| f(t, x)||_{L^{\infty}(0, T; \calH)}\leq   e^{\frac{C_3}{T}} ||y_0||_{\calH}.  \notag
\end{equation} 
\end{theorem}
\begin{proof}
 For the ease of presentation, we only consider the case when $1/T= 2^{n_0}$ with $n_0\in N^*$.  The other cases can  be  treated via  time transition, $i. e. $ if $T\in (2^{-m-1}, 2^{-m})$ then  we simply let the feedback law $U(t;y):=0$ on the time interval  $[2^{-m-1}, T]$.  More precisely, we consider the following  partition of $[0, T]$ and piecewise feedback laws,
\begin{gather}\label{def-In-t}
T_{n}:= 2^{-n_0}\left(1- \frac{1}{2^n}\right), \; I_n:= [T_n, T_{n+1}), \; \lambda_n:= Q^2 2^{2(n_0+n)} \textrm{ for any }  n\geq 0,\\
\textit{for any $n\geq 0$  we consider the control  (feedback law) as $\mathcal{F}_{\lambda_n}$ on interval $I_n$.}  
\end{gather}

\noindent\textbf{Control design.}
\begin{itemize}
\item[\textit{Step 1.}]  Let the constant  $R_T>0$ be sufficiently small to be fixed later on.   First,  for  $||y_0||_{\calH}\leq R_T$,  on the interval  $I_{0}$ we consider the closed-loop system \eqref{cons1}--\eqref{cons-f}  with feedback law $U:= \mathcal{F}_{\lambda_{0}}$ and  initial data $y(0, x)= y_0(x)$.   Assuming that $R_T\leq r_{\lambda_0}$, then    according to Theorem \ref{thm-rap-sta-li} the closed-loop system has a unique solution $\widetilde{y}|_{\bar{I}_{0}}$ that decays exponentially with decay rate $\lambda_0/4$.
\item[\textit{Step 2.}]  Next, we consider  the closed-loop system  with feedback law $\mathcal{F}_{\lambda_{1}}$ and $y(T_{1}, x):= \widetilde{y}(T_1, x)$ on $I_{1}$.  Again we assume  $||y(T_1)||\leq r_{\lambda_1}$ to find a unique solution $\widetilde{y}|_{\bar{I}_{1}}$ that is exponentially stable.
\item[\textit{Step 3.}] By continuing this procedure on $I_n$ and by always assuming $||y(T_n)||\leq r_{\lambda_n}$, we   find a stable  solution $\widetilde{y}|_{I_n}$. 
 \item[\textit{Step 4.}] We denote this constructed  solution $\widetilde{y}|_{[0, T)}\in C^0([0, T); \calH)$ by $y|_{[0, T)}$.
 \item[\textit{Step 5.}]  For some sufficiently small $R_T$ we  prove that $||y(T_n)||$ is  indeed smaller than $r_{\lambda_n}$ for every $n\in \mathbb{N}$,  and show that the solution tends to zero as   $y(T):= \lim_{t\rightarrow T^-} y(t)=0$.
  \item[\textit{Step 6.}]  Eventually,  thanks to Step 5, $y|_{[0, T]}$  is the unique  solution of the Cauchy problem  \eqref{Cauch-NS}  with the control  term $f$ given by $f|_{I_n}:= \mathcal{F}_{\lambda_{n}} y|_{I_n},  \forall n\geq 0$, which satisfies $y(T)=0$.
 \item[\textit{Step 7.}]  We calculate  precise cost estimates.
\end{itemize}

First we \textit{assume}  that for every $I_n$ the value $||y(T_n)||_{L^2}$ is smaller than $r_{\lambda_n}$,  which,  together with   Theorem \ref{thm-rap-sta-li}, implies that  the solution $y|_{I_n}$ verifies
\begin{align}
||y(t)||_{L^2(\Omega)}&\leq C_1 e^{C_1 Q 2^{n_0+n}} e^{-\frac{Q^2}{4}2^{2(n_0+n)} (t-T_n)} ||y(T_n)||_{L^2(\Omega)}, \; \forall t\in I_n,\label{yt1} \\
||\mathcal{F}_{\lambda_n}y(t)||_{L^2(\Omega)}&\leq C_2 e^{C_2 Q 2^{n_0+n}} e^{-\frac{Q^2}{4}2^{2(n_0+n)} (t-T_n)} ||y(T_n)||_{L^2(\Omega)}, \; \forall t\in I_n. \label{ft1}
\end{align}
Consequently,  for every $n\geq 1$ the value of  $y(T_n)$ is dominated by
\begin{align}\label{es-111}
||y(T_n)||_{L^2(\Omega)}&\leq \left(\prod_{k=0}^{n-1} C_1 e^{C_1\sqrt{\lambda_k}} e^{-\frac{\lambda_k}{4} 2^{-(n_0+k+1)} } \right) ||y_0||_{L^2(\Omega)},  \notag\\
&= \left(\prod_{k=0}^{n-1} C_1 e^{C_1Q 2^{n_0+k}}  e^{-\frac{Q^2}{8} 2^{n_0+k}}\right) ||y_0||_{L^2(\Omega)}, \notag\\
&\leq \left(\prod_{k=0}^{n-1} e^{\frac{Q^2}{64} 2^{n_0+k}}  e^{-\frac{Q^2}{8} 2^{n_0+k}}\right) ||y_0||_{L^2(\Omega)}, \notag\\
&= \left(\prod_{k=0}^{n-1}  e^{-\frac{7Q^2}{64} 2^{n_0+k}}  \right) ||y_0||_{L^2(\Omega)}, \notag\\
&= \exp \left(-\frac{7Q^2}{64}  2^{n_0}(2^n-1)\right) ||y_0||_{L^2(\Omega)}.
\end{align}
Observe that the above inequality also holds for $n=0$.
Furthermore, for  any $n\geq 1$ and  for any $t\in I_n$  the control term is bounded by 
\begin{equation}\label{es-222}
||\mathcal{F}_{\lambda_n}y(t)||_{L^2(\Omega)}\leq C_2 e^{C_2 Q 2^{n_0+n}}  ||y(T_n)||_{L^2(\Omega)}\leq  \exp \left(-\frac{5Q^2}{64}  2^{n_0+n-1}\right) ||y_0||_{L^2(\Omega)},
\end{equation}
Clearly, the right hand side of the inequalities \eqref{es-111} and \eqref{es-222} tend to 0 as $n$ tends to $\infty$.   Therefore, it suffices to prove the assumption  $||y(T_n)||_{L^2}\leq r_{\lambda_n}$ to close the ``bootstrap" and to conclude the null controllability.
By recalling the definitions of $\lambda_n$, $r_{\lambda_n}$, and $Q$ we know that  
\begin{equation*}
e^{-\frac{Q^2}{64} 2^{n_0+n}}\leq (C_2 e^{C_2 Q 2^{n_0+n}})^{-1}= (C_2 e^{C_2 \sqrt{\lambda_n}})^{-1}=  r_{\lambda_n}, \;\forall \; n\in \mathbb{N}.
\end{equation*}
Hence, it  suffices to find some $R_T>0$ such that  
\begin{align}\label{cond-RT}
R_T \exp \left(-\frac{7Q^2}{64}  2^{n_0}(2^n-1)\right) \leq e^{-\frac{Q^2}{64} 2^{n_0+n}}\leq r_{\lambda_n}, \;\forall \; n\in \mathbb{N}.
\end{align}
Thus one can  take 
\begin{equation}
R_T:= e^{-\frac{Q^2}{32} 2^{n_0}}= e^{-\frac{Q^2}{32T}}= e^{-\frac{C_3}{T}} \textrm{ where } C_3= \frac{Q^2}{32}.
\end{equation}

It only remains to estimate the controlling cost.  Thanks to \eqref{es-222} we know that 
\begin{equation*}
||f(t)||_{L^2(\Omega)}\leq ||y_0||_{L^2(\Omega)}, \; \forall \; t\in [T_1, T].
\end{equation*}
As for $t\in [0, T_1)$ and the control $f|_{I_0}(t)$,  we have 
\begin{equation*}
||\mathcal{F}_{\lambda_0}y(t)||_{L^2(\Omega)}\leq C_2 e^{C_2 Q 2^{n_0}}  ||y_0||_{L^2(\Omega)}\leq e^{\frac{Q^2}{64} 2^{n_0} } ||y_0||_{L^2(\Omega)}\leq e^{\frac{C_3}{T}} ||y_0||_{L^2(\Omega)}.
\end{equation*}
 In conclusion,  for any $||y_0||_{\calH}\leq e^{-\frac{C_3}{T}}$,   the constructed solution $y(t, x)$ and control $f(t, x)$ satisfy
 \begin{gather*}
 ||y(t, \cdot)||_{L^2(\Omega)} \textrm{ and }||f(t, \cdot)||_{L^2(\Omega)}\longrightarrow 0^+, \textrm{ as } t\rightarrow T^-, \\
 ||y(t, \cdot)||_{L^2(\Omega)} \textrm{ and }||f(t, \cdot)||_{L^2(\Omega)}\leq e^{\frac{C_3}{T}}||y_0||_{L^2(\Omega)}, \; \forall\; t\in [0, T].
 \end{gather*}
\end{proof}
\begin{remark}\label{rmk-r2}
If we replace the linear feedback laws $\{\mathcal{F}_{\lambda_n} y\}_{n=1}^{\infty}$ by $\{\mathcal{K}_{r_{\lambda_n}}\big(\mathcal{F}_{\lambda_n} y\big)\}_{n=1}^{\infty}$ on interval $I_n$, then similar result holds.  Indeed,  according to Theorem \ref{thm-rap-sta-li-K} it suffices to  find  some initial state such that   for  every  $n\in \mathbb{N}$ the value of $||y(T_n)||$ is smaller than $r^2_{\lambda_n}$. More precisely,  instead of taking some $R_T>0$ that satisfies  \eqref{cond-RT}, one only needs to find  $\widetilde{R}_T:= e^{-\frac{Q^2}{16T}}= e^{-\frac{2C_3}{T}}$ satisfying 
\begin{align*}
\widetilde{R}_T \exp \left(-\frac{7Q^2}{64}  2^{n_0}(2^n-1)\right) \leq e^{-\frac{Q^2}{32} 2^{n_0+n}}\leq r_{\lambda_n}^2, \;\forall \; n\in \mathbb{N},
\end{align*}
to  guarantee that  for every $n\in \mathbb{N}$ we have  $||y(T_n)||_{L^2}\leq r_{\lambda_n}^2$.
\end{remark}

\section{Small-time local stabilization}\label{sec-finite}
As in the preceding section, we only focus  on the case when  $T= 1/2^{n_0}$ with $n_0$ be  integer.   We also  adapt the same construction of $T_n$ and $\lambda_n$ given by  \eqref{def-In-t}  in Section \ref{sec-null}.

\begin{theorem}[Small-time local stabilization of Navier-Stokes equations]\label{semi-stab}
Let $T= 1/2^{n_0}$ with $n_0\in \mathbb{N}^*$.   The following  $T$-periodic feedback law $U(t; y): \R\times \calH\rightarrow \calH$ satisfying \eqref{cond-K-F},
\begin{gather}\label{feed-li-stab}
U\big{|}_{[0, T)\times \calH}(t; y):= 
\mathcal{K}_{r_{\lambda_n}}\left(\mathcal{F}_{\lambda_n} y\right), \; \forall \; y\in \calH,  \forall\; t\in I_n, \forall\;  n\in \mathbb{N},
\end{gather}
is a proper feedback law for  system \eqref{cons1}--\eqref{cons-f}. 
 Moreover, for some effectively computable constant $\Lambda_T$ this feedback law stabilizes system \eqref{cons1}--\eqref{cons-f} in finite time:
\begin{itemize}
\item[(i)] ($2T$ stabilization) $\Phi(2T+ t, t; y_0)=0, \;\;\forall \;t\in \mathbb{R},\; \forall\; ||y_0||_{\calH}\leq \Lambda_T.$
\item[(ii)] (Uniform stability)
For every  $\delta> 0$, there exists an effectively computable $\eta> 0$ such that
\begin{equation*}
\big(||(y_0||_{\calH}\leq \eta\big) \Rightarrow \left(||\Phi(t, t'; y_0)||_{\calH}\leq \delta, \;\forall \;  t'\in \R,\; \forall\; t\in ( t', +\infty) \right).
\end{equation*}
\end{itemize}
\end{theorem}

\begin{proof}[Proof of Theorem \ref{feed-li-stab}]
We mimic the prove of the  finite time stabilization of the  heat equations  \cite{2017-Coron-Nguyen-ARMA,  Xiang-heat-2020}, as relatively standard, see also \cite{Coron-Xiang-2018, 2019-xiang-SICON, 2017-Xiang-SCL} for similar results.   The proof is followed by five steps:
\begin{itemize}
\item[\textit{Step 1.}]  The feedback law $U$ is a  proper feedback law, $i. e.$ for any $y_0\in\calH$ and for any initial time $s\in \R$ there exists a unique global (in time) solution.
\item[\textit{Step 2.}]  Null controllability: $\Phi(T , 0 ; y_0)=0$ for any $y_0$ satisfying $||y_0||_{\calH}\leq \widetilde{R}_T= e^{-\frac{2C_3}{T}}$. Moreover, 
\begin{equation}\label{es-1111}
||\Phi(t, 0; y_0)||_{\calH}\leq e^{\frac{C_3}{T}} ||y_0||_{\calH}, \; \forall \; ||y_0||_{\calH}\leq  e^{-\frac{2C_3}{T}}, \;\forall \;t\in [0, T].
\end{equation}
\item[\textit{Step 3.}]  For any   $\widetilde{\eta}> 0$, there exists some $\varepsilon(\widetilde{\eta})\in (0, \widetilde{\eta})$ such that 
\begin{equation}\label{es-2222}
||\Phi(t, s; y_0)||_{L^2(\Omega)}\leq  \widetilde{\eta}, \;\forall \;||y_0||_{L^2(\Omega)}\leq \varepsilon(\widetilde{\eta}), \;\forall \;s\in [0, T), \;\forall \;t\in [s, T].
\end{equation}
\item[\textit{Step 4.}]  $2T$ stabilization:  $\Phi(2T , s; y_0)=0$,   for any $s\in [0, T)$,  for any $y_0$ satisfying $||y_0||_{\calH}\leq \varepsilon
\left(e^{-\frac{2C_3}{T}}\right)=: \Lambda_T$.
\item[\textit{Step 5.}]  Uniform stability as direct consequence of Step 2--4.
\end{itemize}

Step 1.  It suffices to prove that for any $s\in [0, T)$ the closed-loop system has a unique solution on $[s, T]$.  Indeed, thanks to Theorem \ref{thm-clo-sta-F} there exists a unique solution on $I_n$ for any $I_n$ that intersects with $[s, T)$. Hence we find a unique solution $y$ in $C^0([s, T); \calH)\cap L^2_{loc}(s, T; \calVs)$. Observe that the control (provided by the related feedback law) is smaller than $1$, $i.e. ||f(t, x)||_{L^2(s, T; \calH)}\leq \sqrt{T}$. Theorem \ref{Thm-well-stab} implies that the solution $y$ is indeed in $C^0([s, T]; \calH)\cap L^2(s, T; \calVs)$.  Finally, thanks to Theorem \ref{Thm-well-stab} again, the unique solution $y$  never blow up,
\begin{equation*}
||y(t, x)||_{\calH}^2+ ||\nabla y(t, x)||_{L^2(s, t; L^2)}^2\leq ||y_0||_{\calH}^2+ C_0 (t-s), \;\forall \; t\in (s, +\infty).
\end{equation*}

Step 2.  This step is a consequence of Theorem \ref{thm-null-col-opt} and Remark \ref{rmk-r2}.

Step 3. Thanks to the fact that $||f(t, x)||_{\calH}\leq 1$ and  Theorem \ref{Thm-well-stab}, there exists $\tilde{T}\in (0, T)$ such that 
\begin{equation*}
||\Phi(t, s; y_0)||_{\calH}\leq  \widetilde{\eta}, \;\forall \;||y_0||_{\calH}\leq \widetilde{\eta}/2, \;\forall \;s\in [\tilde{T}, T), \;\forall \;t\in [s, T].
\end{equation*}
Observe that  the feedback law $U$ on $[0, \tilde{T})$ is composed  by finitely many stationary feedback laws on intervals $\{I_n\}$, while, thanks to  Theorem \ref{thm-rap-sta-li-K}, on each  of these intervals $I_n$  the system is locally exponentially stable.  Consequently, there exists some $\varepsilon= \varepsilon(\widetilde{\eta})\in (0, \widetilde{\eta}/2)$ such that 
\begin{equation*}
||\Phi(t, s; y_0)||_{L^2(\Omega)}\leq  \widetilde{\eta}/2,\; \forall \;||y_0||_{L^2(\Omega)}\leq \varepsilon, \;\forall \;s\in [0, \tilde{T}), \;\forall \;t\in [s, \tilde{T}].
\end{equation*}

Step 4  is a trivial combination of   Step 2 and Step 3 by  taking $\varepsilon
\left(e^{-\frac{2C_3}{T}}\right)$. 

Step 5 follows directly from Step 2--4.
 \end{proof}
\noindent\textbf{Acknowledgments.}  The author would like to thank Jean-Michel Coron for having attracted his attention to this problem and for fruitful discussions. He also thanks Emmanuel Trélat for  valuable discussions on this problem.  
\bibliographystyle{plain} 
\bibliography{NS}

\begin{thebibliography}{10}

\bibitem{A-S}
Andrey~A. Agrachev and Andrey~V. Sarychev.
\newblock Navier-{S}tokes equations: controllability by means of low modes
  forcing.
\newblock {\em J. Math. Fluid Mech.}, 7(1):108--152, 2005.

\bibitem{Badra-Takahashi}
Mehdi Badra and Tak\'{e}o Takahashi.
\newblock Stabilization of parabolic nonlinear systems with finite dimensional
  feedback or dynamical controllers: application to the {N}avier-{S}tokes
  system.
\newblock {\em SIAM J. Control Optim.}, 49(2):420--463, 2011.

\bibitem{Barbu-MAMS}
Viorel Barbu, Irena Lasiecka, and Roberto Triggiani.
\newblock Tangential boundary stabilization of {N}avier-{S}tokes equations.
\newblock {\em Mem. Amer. Math. Soc.}, 181(852):x+128, 2006.

\bibitem{Barbu-Shirikyan}
Viorel Barbu, S\'{e}rgio~S. Rodrigues, and Armen Shirikyan.
\newblock Internal exponential stabilization to a nonstationary solution for
  3{D} {N}avier-{S}tokes equations.
\newblock {\em SIAM J. Control Optim.}, 49(4):1454--1478, 2011.

\bibitem{Barbu-Triggiani}
Viorel Barbu and Roberto Triggiani.
\newblock Internal stabilization of {N}avier-{S}tokes equations with
  finite-dimensional controllers.
\newblock {\em Indiana Univ. Math. J.}, 53(5):1443--1494, 2004.

\bibitem{Bastin-Coron-book}
Georges Bastin and Jean-Michel Coron.
\newblock {\em Stability and boundary stabilization of 1-{D} hyperbolic
  systems}, volume~88 of {\em Progress in Nonlinear Differential Equations and
  their Applications}.
\newblock Birkh\"{a}user/Springer, [Cham], 2016.
\newblock Subseries in Control.

\bibitem{Kunisch-NS}
Tobias Breiten, Karl Kunisch, and Laurent Pfeiffer.
\newblock Feedback stabilization of the two-dimensional {N}avier-{S}tokes
  equations by value function approximation.
\newblock {\em Appl. Math. Optim.}, 80(3):599--641, 2019.

\bibitem{CS-Lebeau-2016}
Felipe~W. Chaves-Silva and Gilles Lebeau.
\newblock Spectral inequality and optimal cost of controllability for the
  {S}tokes system.
\newblock {\em ESAIM Control Optim. Calc. Var.}, 22(4):1137--1162, 2016.

\bibitem{Chemin-book}
Jean-Yves Chemin.
\newblock {\em Perfect incompressible fluids}, volume~14 of {\em Oxford Lecture
  Series in Mathematics and its Applications}.
\newblock The Clarendon Press, Oxford University Press, New York, 1998.
\newblock Translated from the 1995 French original by Isabelle Gallagher and
  Dragos Iftimie.

\bibitem{Coron}
Jean-Michel Coron.
\newblock {\em Control and nonlinearity}, volume 136 of {\em Mathematical
  Surveys and Monographs}.
\newblock American Mathematical Society, Providence, RI, 2007.

\bibitem{1996-Coron-Fursikov-RJMP}
Jean-Michel Coron and Andrei~V. Fursikov.
\newblock Global exact controllability of the {$2$}{D} {N}avier-{S}tokes
  equations on a manifold without boundary.
\newblock {\em Russian J. Math. Phys.}, 4(4):429--448, 1996.

\bibitem{Coron-Guerrero}
Jean-Michel Coron and Sergio Guerrero.
\newblock Local null controllability of the two-dimensional {N}avier-{S}tokes
  system in the torus with a control force having a vanishing component.
\newblock {\em J. Math. Pures Appl. (9)}, 92(5):528--545, 2009.

\bibitem{Coron-Lissy}
Jean-Michel Coron and Pierre Lissy.
\newblock Local null controllability of the three-dimensional {N}avier-{S}tokes
  system with a distributed control having two vanishing components.
\newblock {\em Invent. Math.}, 198(3):833--880, 2014.

\bibitem{Coron-Marbach-Sueur}
Jean-Michel Coron, Fr\'{e}d\'{e}ric Marbach, and Franck Sueur.
\newblock Small-time global exact controllability of the {N}avier-{S}tokes
  equation with {N}avier slip-with-friction boundary conditions.
\newblock {\em J. Eur. Math. Soc. (JEMS)}, 22(5):1625--1673, 2020.

\bibitem{Coron-Marbach-Sueur-Zhang}
Jean-Michel Coron, Fr\'{e}d\'{e}ric Marbach, Franck Sueur, and Ping Zhang.
\newblock Controllability of the {N}avier-{S}tokes equation in a rectangle with
  a little help of a distributed phantom force.
\newblock {\em Ann. PDE}, 5(2):Paper No. 17, 49, 2019.

\bibitem{2017-Coron-Nguyen-ARMA}
Jean-Michel Coron and Hoai-Minh Nguyen.
\newblock Null controllability and finite time stabilization for the heat
  equations with variable coefficients in space in one dimension via
  backstepping approach.
\newblock {\em Arch. Ration. Mech. Anal.}, 225(3):993--1023, 2017.

\bibitem{coron-rivas-xiang-kdv-16}
Jean-Michel Coron, Ivonne Rivas, and Shengquan Xiang.
\newblock Local exponential stabilization for a class of {K}orteweg--de {V}ries
  equations by means of time-varying feedback laws.
\newblock {\em Anal. PDE}, 10(5):1089--1122, 2017.

\bibitem{Coron-trelat-2004}
Jean-Michel Coron and Emmanuel Tr\'{e}lat.
\newblock Global steady-state controllability of one-dimensional semilinear
  heat equations.
\newblock {\em SIAM J. Control Optim.}, 43(2):549--569, 2004.

\bibitem{Coron-Xiang-2018}
Jean-Michel Coron and Shengquan Xiang.
\newblock Small-time global stabilization of the viscous {B}urgers equation
  with three scalar controls.
\newblock {\em Preprint, hal-01723188}, 2018.

\bibitem{Ervedoza-Glass-Guerrero}
Sylvain Ervedoza, Olivier Glass, and Sergio Guerrero.
\newblock Local exact controllability for the two- and three-dimensional
  compressible {N}avier-{S}tokes equations.
\newblock {\em Comm. Partial Differential Equations}, 41(11):1660--1691, 2016.

\bibitem{FGIP-2004}
E.~Fern\'{a}ndez-Cara, S.~Guerrero, O.~Yu. Imanuvilov, and J.-P. Puel.
\newblock Local exact controllability of the {N}avier-{S}tokes system.
\newblock {\em J. Math. Pures Appl. (9)}, 83(12):1501--1542, 2004.

\bibitem{Fursikov-stab}
A.~V. Fursikov.
\newblock Stabilizability of two-dimensional {N}avier-{S}tokes equations with
  help of a boundary feedback control.
\newblock {\em J. Math. Fluid Mech.}, 3(3):259--301, 2001.

\bibitem{Fursikov-1996}
A.~V. Fursikov and O.~Yu. \`Emanuilov.
\newblock Exact local controllability of two-dimensional {N}avier-{S}tokes
  equations.
\newblock {\em Mat. Sb.}, 187(9):103--138, 1996.

\bibitem{Fursikov-Imanuvilov-book-1997}
Andrei~V. Fursikov and Oleg~Yu. Imanuvilov.
\newblock {\em Controllability of evolution equations}, volume~34 of {\em
  Lecture Notes Series}.
\newblock Seoul National University, Research Institute of Mathematics, Global
  Analysis Research Center, Seoul, 1996.

\bibitem{Hayat-2019}
Amaury Hayat.
\newblock Boundary stability of 1-{D} nonlinear inhomogeneous hyperbolic
  systems for the {$C^1$} norm.
\newblock {\em SIAM J. Control Optim.}, 57(6):3603--3638, 2019.

\bibitem{Lebeau-Robbiano-CPDE}
Gilles Lebeau and Luc Robbiano.
\newblock Contr\^ole exact de l'\'equation de la chaleur.
\newblock {\em Comm. Partial Differential Equations}, 20(1-2):335--356, 1995.

\bibitem{Leray}
Jean Leray.
\newblock Sur le mouvement d'un liquide visqueux emplissant l'espace.
\newblock {\em Acta Math.}, 63(1):193--248, 1934.

\bibitem{Lions-Hilbert}
Jacques-Louis Lions.
\newblock {\em Contr\^{o}labilit\'{e} exacte, perturbations et stabilisation de
  syst\`emes distribu\'{e}s. {T}ome 2}, volume~9 of {\em Recherches en
  Math\'{e}matiques Appliqu\'{e}es [Research in Applied Mathematics]}.
\newblock Masson, Paris, 1988.
\newblock Perturbations. [Perturbations].

\bibitem{Lions-Zuazua}
Jacques-Louis Lions and Enrique Zuazua.
\newblock Exact boundary controllability of {G}alerkin's approximations of
  {N}avier-{S}tokes equations.
\newblock {\em Ann. Scuola Norm. Sup. Pisa Cl. Sci. (4)}, 26(4):605--621, 1998.

\bibitem{Raymond-2006}
J.-P. Raymond.
\newblock Feedback boundary stabilization of the three-dimensional
  incompressible {N}avier-{S}tokes equations.
\newblock {\em J. Math. Pures Appl. (9)}, 87(6):627--669, 2007.

\bibitem{Raymond-2006-2}
Jean-Pierre Raymond.
\newblock Feedback boundary stabilization of the two-dimensional
  {N}avier-{S}tokes equations.
\newblock {\em SIAM J. Control Optim.}, 45(3):790--828, 2006.

\bibitem{Schmidt-Trelat}
Michael Schmidt and Emmanuel Tr\'{e}lat.
\newblock Controllability of {C}ouette flows.
\newblock {\em Commun. Pure Appl. Anal.}, 5(1):201--211, 2006.

\bibitem{2017-Xiang-SCL}
Shengquan Xiang.
\newblock Small-time local stabilization for a {K}orteweg-de {V}ries equation.
\newblock {\em Systems \& Control Letters}, 111:64 -- 69, 2018.

\bibitem{2019-xiang-SICON}
Shengquan Xiang.
\newblock Null controllability of a linearized {K}orteweg-de {V}ries equation
  by backstepping approach.
\newblock {\em SIAM J. Control Optim.}, 57(2):1493--1515, 2019.

\bibitem{Xiang-heat-2020}
Shengquan Xiang.
\newblock Quantitative rapid and finite time stabilization of the heat
  equation.
\newblock {\em Preprint}, 2020.

\bibitem{ZhangRapidStab}
Christophe Zhang.
\newblock {Internal rapid stabilization of a 1-D linear transport equation with
  a scalar feedback}.
\newblock Preprint, October 2018.

\end{thebibliography}

\end{document}